\documentclass[12pt, a4paper,reqno,oneside]{amschanged}

\usepackage{hyperref}
\usepackage{amsmath}
\usepackage{amssymb,stmaryrd}
\usepackage{amsthm}
\usepackage{psfrag}
\usepackage[utf8]{inputenc}
\usepackage{graphicx}
\usepackage{hyperref}
\usepackage{mathtools}
\usepackage{calrsfs}
\usepackage{tikz}
\usepackage{nicefrac}
\usepackage{graphicx}
\usepackage{caption}
\usepackage{subcaption}

\usetikzlibrary{decorations.pathmorphing, patterns, calc,snakes}
\definecolor{dblue}{rgb}{0.09,0.32,0.44} 
\hypersetup{
    colorlinks,%
    citecolor=dblue,%
    filecolor=black,%
    linkcolor=dblue,%
    urlcolor=blue,
    pdfborder={0 0 0},
    pdfborderstyle={}
}

\usepackage{fullpage}
\newtheorem{theorem}{Theorem}

\newtheorem{proposition}[theorem]{Proposition}
\newtheorem{lemma}[theorem]{Lemma}

\theoremstyle{remark}

\newtheorem*{remark*}{Remark}

\newcommand{\lr}[1][]{\xleftrightarrow{\: #1 \:\:}}

\newcommand\fullsizenesearrows{\ensuremath{\setlength{\unitlength}{1cm}
\begin{picture}(0.3,0.34)
\put(0,0.17){{$\scriptscriptstyle\nearrow$}}
\put(0,0){{$\scriptscriptstyle\searrow$}}
\end{picture}}}

\newcommand\nesearrows{{\mathchoice
  {\fullsizenesearrows}
  {\fullsizenesearrows}
  {\scalebox{0.65}{\fullsizenesearrows}}
  {\scalebox{0.4}{\fullsizenesearrows}}
}}

\usepackage{titlesec}
\titleformat{\section}{\Large\bfseries}{\thesection}{1em}{}
\titleformat{\subsection}{\bfseries}{\thesubsection}{1em}{}

\usepackage{courier}

\usepackage{array}
\newcolumntype{e}{>{\displaystyle}r @{\,} >{\displaystyle}c @{\,} >{\displaystyle}l}

\newcounter{constant}

\newcommand{\ep}{\varepsilon}
\begin{document}


\title{\LARGE \usefont{T1}{tnr}{b}{n} \ \rlap{\lowercase{\begin{picture}(0,0)\put(-20,-20){\includegraphics[scale=0.32]{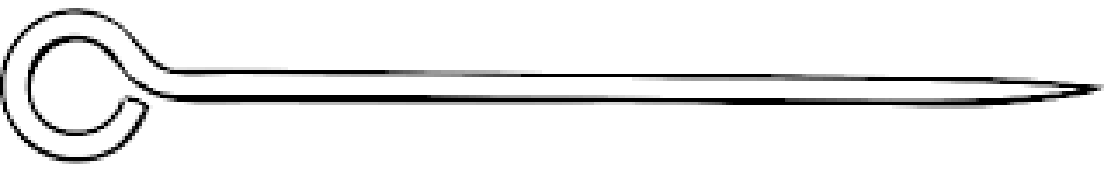}}\end{picture}}}B\lowercase{rochette  percolation}}


\author{\color{black} \normalsize \itshape H. D\lowercase{uminil}-C\lowercase{opin} $^1$ \color{white} \tiny and}
\address{$^1$ Universit\'e de Gen\`eve, Section de Math\'ematiques \newline \hspace*{6mm} 
Gen\`eve 1211, Switzerland
{\itshape \texttt{hugo.duminil@unige.ch}}.}

\author{\color {black} \normalsize \itshape M. R. H\lowercase{il\'ario} $^1$ $^2$ \color{white} \tiny and}
\address{$^2$ Universidade Federal de Minas Gerais, Departamento de Matem\'atica, 
\newline \hspace*{6mm} Belo Horizonte 31270-901, Brazil 
{\itshape \texttt{mhilario@mat.ufmg.br}}.}

\author{\color{black} \normalsize \itshape G. K\lowercase{ozma} $^3$ \color{white} \tiny and}
\address{$^3$ The Weizmann Institute of Science 
\newline \hspace*{6mm} Rehovot 7610001, Israel
{\itshape \texttt{gady.kozma@weizmann.ac.il}}.}

\author{\color{black} \normalsize \itshape V. S\lowercase{idoravicius} $^4$  $^5$ \color{white} \tiny}
\address{$^4$    
Courant Institue of Mathematical Sciences, NYU,
\newline \hspace*{6mm} 251 Mercer Street, New York, NY 10012, USA, and}
\medskip
\address{$\; \, $ NYU-ECNU Institute of Mathematical Sciences at NYU Shanghai,
\newline \hspace*{6mm} 3663 Zhongshan Road North, Shanghai, 200062, China.
{\itshape \texttt{vs1138@nyu.edu}}.}

\address{$^5$ Cemaden, 
Estrada Dr. Altino Bondensan, 500 Coqueiro, 
\newline \hspace*{6mm} São José dos Campos - SP,  12247-016, Brasil.
}

\date{\today}

\begin{abstract}
We study bond percolation on the square lattice with one-dimensional inhomogeneities. Inhomogeneities are introduced 
in the following way: 
A vertical column on the square lattice is the set of vertical edges that project to the same vertex on $\mathbb{Z}$.
Select vertical columns at random independently with a given positive probability.
Keep (respectively remove) vertical edges in the selected columns, with probability $p$, (respectively $1-p$).
All horizontal edges and vertical edges lying in unselected columns are kept (respectively removed) with probability $q$,  (respectively $1-q$).
We show that, if $p > p_c(\mathbb{Z}^2)$ (the critical point for homogeneous Bernoulli bond percolation) then $q$ can be taken strictly smaller then $p_c(\mathbb{Z}^2)$ in such a way that the probability that the origin percolates is still positive.
 
\end{abstract}

\maketitle

\section{Introduction}\label{sec:intr}

\subsection{Definition of the model and statement of the result}

Consider the square lattice $\mathbb{Z}^2=(V(\mathbb{Z}^2),E(\mathbb{Z}^2))$ defined by
\begin{align*}V(\mathbb Z^2)&:=\big\{x=(x_1,x_2)\in\mathbb R^2:x_1,x_2\in\mathbb Z\big\},\\
E(\mathbb Z^2)&:=\big\{\{x,y\}\subset V(\mathbb{Z}^2):|x_1-y_1|+|x_2-y_2|=1\big\}.\end{align*}
A {\em percolation configuration} is an element of $\{0,1\}^{E(\mathbb Z^2)}$ denoted generically by $\omega=(\omega(e):e\in E(\mathbb Z^2))$. Note that $\omega$ can be seen as a subgraph of $\mathbb Z^2$ by setting 
$V(\omega):=V(\mathbb Z^2)$ and $E(\omega):=\{e\in E(\mathbb Z^2):\omega(e)=1\}$.
The study of the connectivity properties of this subgraph obtained when $\omega$ is sampled at random is the main goal of percolation theory.

In the Bernoulli bond percolation model on $\mathbb{Z}^2$, the $\omega(e)$'s are independent Bernoulli random variables with mean $p_e$.
The model is said to be homogeneous when, for every $e$, $p_e = p$ for some $p \in [0,1]$. 
Otherwise it is said to be inhomogeneous.
One way of introducing inhomogeneities is by modifying the `weight' $p_e$ of edges $e$ lying along a fixed set of vertical columns.

There are also several important dependent percolation models in which, in contrast to Bernoulli percolation, the state of the edges are not independent in contrast to Bernoulli percolation.
In this paper we study one such model where the state of vertical edges lying in the same vertical column are correlated as we describe now.

For $\Lambda\subset\mathbb Z$, set 
$$E_{\rm ver}(\Lambda\times\mathbb Z):=\big\{\{(x_1,x_2),(x_1,x_2+1)\}:x_1\in\Lambda, x_2\in \mathbb Z\big\}.$$ 
For $p,q\in[0,1]$, let $\mathbb P^\Lambda_{p,q}$ be the law on $\{0,1\}^{E(\mathbb Z^2)}$ under which the $\omega_e$ ($e\in E(\mathbb Z^2)$) are independent Bernoulli random variables with mean $p_e$ given by 
$$p_e=\begin{cases}p&\text{ if }e\in E_{\rm ver}(\Lambda\times\mathbb Z)\\ q&\text{ if }e\notin E_{\rm ver}(\Lambda\times\mathbb Z)\end{cases}.$$
Notice that $\mathbb{P}^{\Lambda}_{p,p}$ is the {\em Bernoulli bond percolation} measure with edge-weight $p$, which will be denoted by $\mathbb{P}_p$.
However, when $p \neq q$ and $\Lambda$ is non-empty, $\mathbb{P}^{\Lambda}_{p,q}$ is a inhomogeneous Bernoulli percolation due to the presence of one-dimensional columnar inhomogeneities  along the vertical columns that project to $\Lambda$.
Also, in this case, the model is no longer translation invariant.

We now wish to take $\Lambda$ at random. 
For each $\rho \in [0,1]$, define $\nu_{\rho}$ to be the probability measure on subsets of $\mathbb{Z}$ under which $\{i\in\Lambda\}$ are independent events having probability $\rho$.
We are now in a position to state our main result.
Denote by $\{0 \longleftrightarrow \infty\}$ the event that the origin is connected to infinity (see Section \ref{s:notation} for a precise definition). 
Let $p_c$ be such that $\mathbb P_p(0\leftrightarrow \infty)$ equals 0 if $p<p_c$, and is strictly positive if $p>p_c$ (Kesten proved in \cite{Kesten80} that $p_c=1/2$).
\begin{theorem}
\label{t:main}
For every $\varepsilon \in (0,1/2]$ and $\rho >0$, there exists $\delta >0$ such that for $\nu_\rho$-almost every $\Lambda$,
\[\mathbb{P}_{p_c+\varepsilon, p_c-\delta}^{\Lambda}(0\longleftrightarrow\infty) > 0.\]
\end{theorem}

Before we continue, let us just mention a few more words about the nature of this model.
Although, in general, the `quenched' law $\mathbb{P}^{\Lambda}_{p,q}$ is inhomogeneous, the annealed law (i.e.\ the law obtained by averaging $\mathbb{P}^{\Lambda}_{p,q}(\cdot)$ over the realisations of $\Lambda$) is homogeneous: Each edge has the same weight.
Furthermore the annealed law is translation invariant, since $\nu_{\rho}$ also is.
However, it is a dependent percolation since the correlation between the state of edges lying on a same vertical column is a positive constant that does not even decay with their distance.

In the remainder of this section we present some of our motivations for addressing this problem, mention some related works, and highlight the ideas and techniques to  be used in the proof of the above theorem.
\medskip

\subsection{Motivation and related models}

This work is motivated by the following general question: how do $d$-dimensional inhomogeneities in $(d+1)$-dimensional lattice models shift the critical point or change the order of the phase transition?
This question was raised before for a number of models and settings as we describe below.

In the context of percolation, a classical argument due to Aizenman and Grimmett \cite{Aizenman_Grimmett} guarantees the validity of Theorem \ref{t:main} in the case that $\Lambda$ only contains bounded gaps (i.e.\ when there exists a $k \in (0, \infty)$ such that $\Lambda$ intersects all sets of the type $[l,r] \cap \mathbb{Z}$ with $r-l = k)$.
Note that, for any $\rho \in (0,1)$, the set of $\Lambda$'s that exhibit such a regularity condition, has zero measure under $\nu_{\rho}$.
We will use the framework developed in \cite{Aizenman_Grimmett} as one of the elements of our proof.

In \cite{Zhang}, Zhang addresses the case when $\Lambda = \{0\}$, and $q= p_c$.
Relying on the idea of Harris \cite{Harris} of constructing dual circuits around the origin together with the Russo \cite{Russo78} and Seymour and Welsh \cite{Seymour_Welsh} techniques, he proves that $\mathbb{P}^{\{0\}}_{p, p_c} (0 \leftrightarrow\infty)=0$ for any $p\in [0,1)$. 
Monotonicity implies that $\mathbb{P}^{\{0\}}_{p,q} (0\leftrightarrow\infty) =0$ if $q < p_c$, and the results of Barsky, Grimmett and Newman on percolation in half-spaces \cite{Barsky_Grimmett_Newman} imply immediately that $\mathbb{P}^{\{0\}}_{p,q} (0\leftrightarrow\infty) >0$ whenever $q > p_c$.
On the other extreme, when $\Lambda = \mathbb{Z}$  classical arguments due to Kesten show that $\mathbb{P}_{p,q}^{\mathbb{Z}} (0 \leftrightarrow \infty) >0$ iff $p+q >1$ (see page 54 in \cite{Kesten82} or Section 11.9 in \cite{Grimmett}).

For the Ising model on the square lattice, McCoy and Wu \cite{McCoy_Wu} considered the setting in which the coupling constants for horizontal edges are given by a fixed deterministic number whereas for vertical edges, all the coupling constants for edges connecting between sites in the $j^\textrm{th}$ and $j+1^\textrm{st}$ rows are given by a random variable $E(j)$.
There, the $E(j)$'s are assumed to be i.i.d.
Their main motivation was to show how the presence of inhomogeneities leads to a model where the specific heat does not diverge (and not even its derivatives) close to the critical temperature.
This contrasts with the classical results of Onsager for the Ising model on the square lattice with homogeneous coupling constants.

In \cite{Campanino_Klein}, motivated by the study of the Ising model with a random transverse field, Campanino and Klein studied the decay of the two-point function for  a $(d+1)$-dimensional bond percolation (and also Ising and Potts models) with both $d$-dimensional and $1$-dimensional disorder.

Another variation was studied by Hoffman \cite{Hoffman}. 
That paper discussed percolation in a random environment where columns of horizontal edges were weakened independently, and so were rows of vertical edges.

More recently, in \cite{Kesten_Sidoravicius_Vares}, the authors considered the $(1+1)$-directed percolation model with inhomogeneities that are transversal to the ``time direction'' (in contrast with analogous results on the contact process in \cite{Bramson_Durrett_Schonmann, Madras_Schinazi_Schonmann} where the inhomogeneities are taken along lines parallel to the ``time direction'').
We discuss their setting and their result in more detail in Section \ref{sec:outline} since they are used as a fundamental step in our work (see Theorem \ref{thm:main input} below).

\subsection{Notation}
\label{s:notation}

When there is no risk of ambiguity, we abuse notation and do not distinguish between $V(\mathbb{Z}^2)$ and $\mathbb{Z}^2$, and similarly for other graphs.
Let us write $x\sim y$ if $x$ is a {\em neighbour} of $y$ i.e.\ if $\{x,y\}\in E(\mathbb Z^2)$. For $A\subset \mathbb Z^2$, we set $\partial A:=\{x\in A:\exists\, y\notin A\text{ with }x \sim y\}$.
A path in $A$ is a sequence of sites $v_0 \sim v_1 \sim\cdots \sim v_n$ such that $v_i \in A$ for all $i$.
 \medbreak
An edge $e$ is said to be {\em open} (in $\omega$) if $\omega(e)=1$. Otherwise, it is said to be {\em closed}.
For a set $A\subset \mathbb Z^2$ and two vertices $x,y\in\mathbb Z^2$, $x$ and $y$ are {\em connected in $A$} (denoted $x\stackrel{A}{\longleftrightarrow} y$) if there exists a sequence $x=v_0\sim\dots \sim v_n=y$ in  $A$ such that $\omega(\{v_i,v_{i+1}\})=1$ for every $0\le i<n$.
If $A=\mathbb Z^2$, we omit it from the notation and simply say that $x$ and $y$ are connected.
The cluster of a site $x$ in a set $A$ is the set of all sites $y$ for which $\omega \in \{x \stackrel{A}{\longleftrightarrow} y\}$.
We denote by $\{x \longleftrightarrow \infty\}$ the event that there exists an unbounded sequence $(y_n) \subset \mathbb{Z}^2$ such that $x$ is connected to each one of the $y_n$'s.
\medbreak
For $n < m\in\mathbb Z$, set $\llbracket n,m\rrbracket=\{n,n+1,\dots,m-1,m\}$.  
A set of this form will be called an interval of $\mathbb{Z}$ and its diameter is defined to be equal to $m-n$.
For $n\ge1$ and $x\in\mathbb Z^2$, set $B_n(x):=x+\llbracket -n,n\rrbracket^2$, the box of size $n$ centred at $x$. 
Define $\mathcal A_n(x)$ to be the event that there exists an open circuit in $B_{2n-1}(x)$ surrounding $B_n(x)$, i.e.\ that there exists a path $v_0\sim v_1\sim\dots\sim v_k= v_0$ such that:
\begin{itemize}
\item For all $0 \leq i \leq k$, $v_i$ belongs to $B_{2n-1}(x) \backslash B_{n}(x)$;
\item For all $0\le i<k$, $\omega(\{v_i,v_{i+1}\})=1$;
\item The winding number of the path around $x$ is non-zero.
\end{itemize}
If $x=0$, we simply write $\mathcal A_n$ instead of $\mathcal A_n(0)$.

In what follows, we denote by $c$ a generic strictly positive constant whose value may change at each appearance.
A numbered constant such as $c_1, c_2, \ldots$ will have their value fixed at its first appearance.

\subsection{Summary of the proof}\label{sec:outline}

Let us start by recalling the results of
\cite{Kesten_Sidoravicius_Vares} and comparing them to ours. The
problem analysed in \cite{Kesten_Sidoravicius_Vares} differs from ours
in the order of quantifiers (and also in the choice of the two-dimensional
lattice). Our result is that even if the ``strong columns'' are
rare and just slightly strong, they still allow percolation. The result of
\cite{Kesten_Sidoravicius_Vares} is that even if the ``weak
columns'' are very weak, if they are sufficiently rare they do not
disrupt percolation. 

Our proof strategy is to reduce our problem to that of
\cite{Kesten_Sidoravicius_Vares} using a one-step renormalisation
procedure. This means that we find some $n$ such that columns of width
$n$ are ``good'' with high probability, and inside each good column,
each $n\times n$ block is good with high probability, while inside a
bad column, each block is good with probability bigger than some
constant independent of $n$. This will allow us to show that our
renormalised model stochastically dominates that of
\cite{Kesten_Sidoravicius_Vares}, and hence percolates.

The choice of $n$ is probably the interesting part in the procedure
and requires some knowledge of near-critical percolation. It would be
interesting to generalise our results to 3 dimensions, but our
understanding of near-critical 3-dimensional percolation currently
falls short of what is needed for the result. On the other hand, the
results of \cite{Kesten_Sidoravicius_Vares}, which are strictly
2-dimensional, are not necessary in the 3-dimensional case. It is only
the near-critical behaviour that is missing.

We will now describe the renormalisation procedure, and then return to
\cite{Kesten_Sidoravicius_Vares} and state their result in details.
%
The first step is to compare $\mathbb{P}^{\Lambda}_{p,q}$ for
$\Lambda$ without big gaps to near-critical percolation. 
More precisely, given an integer $k\ge 1$, a subset $\Lambda \subset \mathbb{Z}$ is called {\em $k$-syndetic} if it intersects all intervals of $\mathbb Z$ having diameter $k$.
The following proposition shows that, starting from critical percolation, the effect of enhancing the parameter on $E_{\text{vert}}(\Lambda \times \mathbb{Z})$ for a $k$-syndetic set $\Lambda$ is comparable to the effect of performing a certain homogeneous sprinkling.
\begin{proposition}
\label{l:Pp_c}
Let $\varepsilon \in (0, 1/2)$. There exists $c_1>0$ such that for any
$k$ large enough (depending on $\varepsilon$),
\begin{align}
\label{e:spread2}
\mathbb{P}_{p_c+\varepsilon, p_c}^{\Lambda} (\mathcal A_n)&\ge \mathbb{P}_{p_c + k^{-c_1}} (\mathcal A_n)
\end{align}
for any $k$-syndetic $\Lambda$ and any $n\ge k$.
\end{proposition}
The proof is based on some quantitative estimates for non-local and non-translation invariant versions of enhancements. 
Local and translation invariant enhancements were studied by Aizenman and Grimmett in \cite{Aizenman_Grimmett}.


Let $\varepsilon \in (0,1/2)$ be fixed.
We wish to prove that for any $c_2>0$ and $\delta>0$, there exists $n$ large enough such that for any  $(c_2\log n)$-syndetic set $\Lambda$,
\begin{align}
\label{e:spread3}
\mathbb{P}_{p_c+\varepsilon, p_c}^{\Lambda} (\mathcal A_n)&\ge 1-\delta.
\end{align}
In order to do that, we invoke general statements coming from the theory of near-critical percolation to prove the following proposition, which together with Proposition~\ref{l:Pp_c}, implies \eqref{e:spread3}.
\begin{proposition}
\label{l:limpp_c}
For any $c_3>0$,  we have
$$\lim_{n\to\infty} \mathbb{P}_{p_c + (\log n)^{-c_3}} (\mathcal A_n)=1.$$
\end{proposition}

We now return to the results of \cite{Kesten_Sidoravicius_Vares}. As
already mentioned, they are on a different 2-dimensional (directed) lattice which we describe next. 

Let $\nesearrows$ denote the lattice with sites given by $V(\nesearrows) = \{ x= (x_1,x_2) \in \mathbb{Z} \times \mathbb{Z}:\, x_1 + x_2 \text{ is even}\}$ and oriented edges $E(\nesearrows) = \{ [x,y] \subset V(\nesearrows): y_1 - x_1 = 1 \text{ and }  |x_2 - y_2| = 1 \}$.
Note that only edges oriented in the \emph{north-east} or \emph{south-east} direction are allowed.
As before we denote $x \sim y$ if $[x,y] \in E(\nesearrows)$.

A \emph{column} is a set of the type ${c}(i) = \{(i,j) \in V(\nesearrows); \, j \in \mathbb{Z}\}$. 
Fix $p_B$, $p_G$ and $\rho'$ in the interval $(0,1)$.
Let $\widetilde{\Lambda}\subset\mathbb{Z}$ be a random set such that the events $\{i\in\widetilde\Lambda\}$ are i.i.d.\ with probability $\rho'$ and declare the column ${c}(i)$ to be \emph{good} if $i\in\widetilde\Lambda$.
Columns that are not good are called \emph{bad} columns.
Conditionally on the state of the columns, we then declare each site in bad columns to be occupied or vacant with probability $p_B$ and $1-p_B$, respectively.
Similarly we declare each site in a good column to be occupied or vacant with probability $p_G$ and $1-p_G$, respectively.
Conditioned on $\tilde{\Lambda}$, the state of each site is decided independently of the others.
We denote by $\tilde{\mathbb{P}}^{\tilde{\Lambda}}_{p_B,p_G}$ the law in $\{0,1\}^{V(\nesearrows)}$ conditional on the state of the columns $\widetilde{\Lambda}$.

For a configuration $\omega \in \{0,1\}^{V(\nesearrows)}$,
we say that the origin belongs to an infinite connected component for oriented percolation in $\nesearrows$ if there exists an infinite sequence $0 = v_0 \sim v_1 \sim v_2 \sim \cdots$ with $v_i \neq v_j$ when $i\neq j$ and such that $\omega(v_i)=1$ for all $i \geq 0$.
The critical percolation on $\nesearrows$ will be denoted by $p_c(\nesearrows)$.

We are ready to state the main input to our renormalisation scheme.

\begin{theorem}[Kesten, Sidoravicius, Vares \cite{Kesten_Sidoravicius_Vares}]\label{thm:main input}
Assume that $p_B >0$ and that $p_G> p_c(\nesearrows)$.
Then, there exists a $\rho'<1$ such that, for almost all realisations of $\widetilde{\Lambda}$,
\[
\tilde{\mathbb{P}}^{\tilde{\Lambda}}_{p_B,p_G}  (0 \text{ belongs to an infinite oriented connected component}) > 0.
\]
\end{theorem}

We finally get to the renormalisation scheme. 
Let $n$ be an integer. 
We think about $2n\nesearrows$ as a subset of $\mathbb Z^2$ and examine the events $\mathcal A_n(2nv)$ for $v\in\nesearrows$ (see Figure \ref{fig:L} below).
We say that the $i^\textrm{th}$ column of $\nesearrows$ is {\em good} if $\Lambda$ intersects every subinterval of $\llbracket 2n(i-1), 2n(i+1)\rrbracket$ that has diameter $\lceil \frac{2}{\rho} \log(2n) \rceil$. Otherwise, the column is said to be {\em bad}.

Now, $v\in \nesearrows$ is said to be {\em occupied} if $\mathcal A_n(2nv)$ occurs.  On the one hand, for $v$ in a bad column, classic crossing estimates at criticality imply that the probability of such $v$ being occupied is larger than some constant $c>0$ independent of $n$. On the other hand, for $v$ in a good column, \eqref{e:spread3} implies that the probability of being occupied can be made as close to 1 as we wish, provided that $n$ is chosen large. Denote $X(v)=\mathbf{1}[\mathcal A_n(2nv)]$ for brevity.

Note that $X(v)$ and $X(w)$ are not
independent if $v$ and $w$ are neighbours in $\nesearrows$. They are
only $1$-dependent i.e.\ each $X(v)$ is independent of
$\{X(w):|v-w|>2\}$. 
Similarly, the events that columns $i$
and $i+1$ are good are not independent. Nevertheless, one may compare
these $1$-dependent events with independent percolation using
standard methods such as Liggett-Schonmann-Stacey
\cite{Liggett_Schonmann_Stacey}.

The renormalisation scheme is now clear: By choosing $n$ large enough,
one may guarantee that each column is good with probability close to
1, and that every vertex in a good column is occupied with good
probability, move from $1$-independent events to truly independent
events, and then apply Theorem \ref{thm:main input}. The details
occupy the remainder of the paper.

\section{Crossing estimates for \texorpdfstring{$k$}{k}-syndetic sets}\label{sec:2}

In this section we prove Proposition \ref{l:Pp_c} (which states that $\mathbb{P}^{\Lambda}_{p_c+\varepsilon, p_c} (\mathcal{A}_n) \geq \mathbb{P}_{p_c+k^{-c_1}} (\mathcal{A}_n)$). 
We start with the approach of Aizenman and Grimmett \cite{Aizenman_Grimmett} which allows to reduce the problem to a problem about comparison of pivotality probabilities. In other words, to reduce Proposition \ref{l:Pp_c} to the following.

\begin{proposition}
\label{l:russo}
Let $\ep \in (0,1/2)$. There exists $c_4>0$ such that for any $k$
large enough, any $k$-syndetic set $\Lambda\subset\mathbb{Z}$, any $n \geq k$ and any $(p,q) \in [p_c, p_c+\varepsilon] \times [p_c-k^{-2}, p_c+k^{-2}]$,
\begin{equation}\label{eq:spread4}
 \frac{\partial }{\partial q}\mathbb{P}_{p,q}^{\Lambda} (\mathcal A_n)\le k^{c_4}\cdot\frac{\partial}{\partial p}\mathbb{P}_{p,q}^{\Lambda} ( \mathcal A_n).
\end{equation}
\end{proposition}
Before proving this result, let us show how it implies Proposition~\ref{l:Pp_c} (as mentioned above, our argument is similar to the one in \cite{Aizenman_Grimmett}).
\begin{proof}[Proof of Proposition~\ref{l:Pp_c}]
Choose $c_1>\max\{c_4,2\}$ and let $k$ be large enough so that $k^{-c_1}<\min\{k^{-2},\ep/(2k^{c_4})\}$ and that the previous proposition applies.
Let $\Lambda$ be a $k$-syndetic set and $n \geq k$. For any $t\in [0,1]$, let us define
\[
p(t) = p_c + (1-t)k^{-c_1} + t \ep \qquad \text{and} \qquad q(t) = p_c + (1-t)k^{-c_1}.\]
With the notation $f(p,q):=\mathbb{P}_{p,q}^{\Lambda} (\mathcal A_n)$,
which is a polynomial in $p$ and $q$ and in particular
differentiable, we find
\begin{align}
\label{e:ddt}
\frac{d}{dt} f(p(t),q(t))&= p'(t) \frac{\partial}{\partial p} f(p(t),q(t)) + q'(t) \frac{\partial}{\partial q}f(p(t),q(t))\\
&= \left(-k^{-c_1} + \ep \right) \frac{\partial}{\partial p}f(p(t),q(t))  - k^{-c_1} \frac{\partial}{\partial q}f(p(t),q(t)) .
\end{align}
Since $(p(t),q(t)) \in [p_c, p_c+\varepsilon] \times [p_c, p_c+k^{-2}]$ for all $t \in [0,1]$, Proposition \ref{l:russo} implies
\[
\frac{d}{dt} f(p(t),q(t)) \geq \big(-k^{-c_1} + \ep -k^{c_4-c_1}\big)\, \frac{\partial}{\partial p} f(p(t),q(t))\ge 0
\]
from which we conclude $f(p(0),q(0)) \leq f(p(1),q(1))$, a fact which gives \eqref{e:spread2}.
\end{proof}

We now focus on the proof of Proposition~\ref{l:russo}.
We will need the notion of dual configuration. 
Let $(\mathbb Z^2)^*=(\tfrac12,\tfrac12)+\mathbb Z^2$. 
Vertices and edges of $(\mathbb Z^2)^*$ are called dual vertices and dual edges. 
Each edge $e$ of $\mathbb Z^2$ corresponds to a dual edge $e^*$ of $(\mathbb Z^2)^*$ that it intersects in its middle.
As before, we write $u \sim v$ if $u$ and $v$ are endpoints of a dual edge.
Also define $B^*$ to be the subset of $(\mathbb Z^2)^*$ of endpoints of dual edges of the form $\{x,y\}^*$ with $x,y\in B$.

Define the dual configuration $\omega^*\in\{0,1\}^{E((\mathbb Z^2)^*)}$ of $\omega\in\{0,1\}^{E(\mathbb Z^2)}$ by $\omega^*(e^*)=1-\omega(e)$. 
Two dual vertices $u$ and $v$ of $(\mathbb Z^2)^*$ are {\em dual-connected} in $V\subset (\mathbb Z^2)^*$ if there exists $u=v_0\sim\dots\sim v_k=v$ such that $v_i  \in V$ for every $0 \leq i \leq k$ and $\omega^*(\{v_i,v_{i+1}\})=1$ for every $0\le i<k$. 
We denote this event by $u\lr[*,V] v$.

\begin{proof}[Proof of Proposition~\ref{l:russo}]
Let $\Lambda$ be $k$-syndetic for some $k \geq 100$ and $(p,q) \in [p_c, p_c+\varepsilon]\times [p_c-k^{-2}, p_c + k^{-2}]$.

Define $E$ to be the set of edges of $\mathbb Z^2$ with both endpoints in $B_{2n-1}\setminus B_n$.
Define also $F=E\cap E_{\rm vert}(\Lambda\times\mathbb Z)$.
For an edge $e=\{x,y\}\in E\setminus F$, let $f(e)\in F$ be a minimiser of the $\|\cdot\|_1$-distance between $\{x,y\}$ and $F$. Note that $f(e)$ may not be defined uniquely (there may be up to six such edges). In case there is more than one choice for $f(e)$, select one of them according to some arbitrary rule. For $f\in F$, let $E(f) = \{e \in E\setminus F: f(e) = f\}$. Russo's Formula (see \cite[Section 2.4]{Grimmett}) implies that
\begin{align*}
\frac{\partial}{\partial q} \mathbb{P}_{p,q}^{\Lambda} (\mathcal A_n) &= \sum_{e \in E\setminus F} \mathbb{P}_{p,q}^{\Lambda} (e \text{ is pivotal for } \mathcal A_n)= \sum_{f \in F} \sum_{e \in E(f)} \mathbb{P}_{p,q}^{\Lambda} (e \text{ is pivotal for } \mathcal A_n).\end{align*}
Now, the fact that $\Lambda$ is $k$-syndetic implies that ${\rm card}(E(f))\le 10k$. If one assumes that there exists $c_5>0$ such that for any $f\in F$ and $e\in E(f)$,
\begin{equation}\label{eq:ll}\mathbb{P}_{p,q}^{\Lambda} (e \text{ is pivotal for } \mathcal A_n)\le k^{c_5}\cdot\mathbb{P}_{p,q}^{\Lambda} (f \text{ is pivotal for } \mathcal A_n),\end{equation}
then we may deduce that for $k$ large enough,
\begin{align*}
\frac{\partial}{\partial q} \mathbb{P}_{p,q}^{\Lambda} (\mathcal A_n)
&\le \sum_{f \in F} {\rm card}(E(f))\cdot\max\big\{\mathbb{P}_{p,q}^{\Lambda} (e \text{ is pivotal for } \mathcal A_n):e\in E(f)\big\}\\
&\le 10k \cdot \sum_{f\in F} \max\big\{\mathbb{P}_{p,q}^{\Lambda} (e \text{ is pivotal for } \mathcal A_n):e\in E(f)\big\}\\
&\le 10 k\cdot k^{c_5} \sum_{f\in F}\mathbb{P}_{p,q}^{\Lambda} (f \text{ is pivotal for } \mathcal A_n)=10k^{c_5+1} \frac{\partial}{\partial p} \mathbb{P}_{p,q}^{\Lambda} (\mathcal A_n),
\end{align*}
where we used Russo's Formula in the last equality. 
This implies the claim with $c_4>c_5+1$ and $k$ large enough. 
\bigbreak
We therefore focus on the proof of \eqref{eq:ll}. 
Fix $f\in F$ and $e=\{x,y\}\in E(f)$. 
By definition of $E(f)$, there exist $z$ and $\ell\le k$ such that $B:=B_\ell(z)$ satisfies (see Figure \ref{fig:eventPab}).
\begin{itemize}
\item $B \subset B_{2n-1}\setminus B_n$,
\item $f$ has both endpoints in $\partial B$,
\item $e$ has both endpoints in $B$,
\item $\Lambda\times\mathbb Z$ does not intersect $B\setminus \partial B$.
\end{itemize}
The proof is going to be based on surgery in the box $B$ (and its immediate neighbourhood). 
For $a,b\in\partial B$, let $\mathcal Q_{a,b}$ be the event that there
is an open path $\gamma$ from $a$ to $b$ in $B_{2n-1}\setminus(B_n\cup B)$
which surrounds $B_n$, or to be more precise, that can be completed to a path surrounding $B_n$ by adding a path from $a$ to $b$ contained in $B$. 
Let also
\begin{align}
\mathcal P_{a,b}&=\{e\text{ is pivotal for }\mathcal A_n\}\cap \mathcal{Q}_{a,b},\\
\label{e:Aab}
\mathcal G_{a,b}&=\{f\text{ is pivotal for
}a\lr[B]b\}\cap\{a\nleftrightarrow \mathbb Z^2\setminus
B\}\cap\{b\nleftrightarrow\mathbb Z^2\setminus B\}.
\end{align}

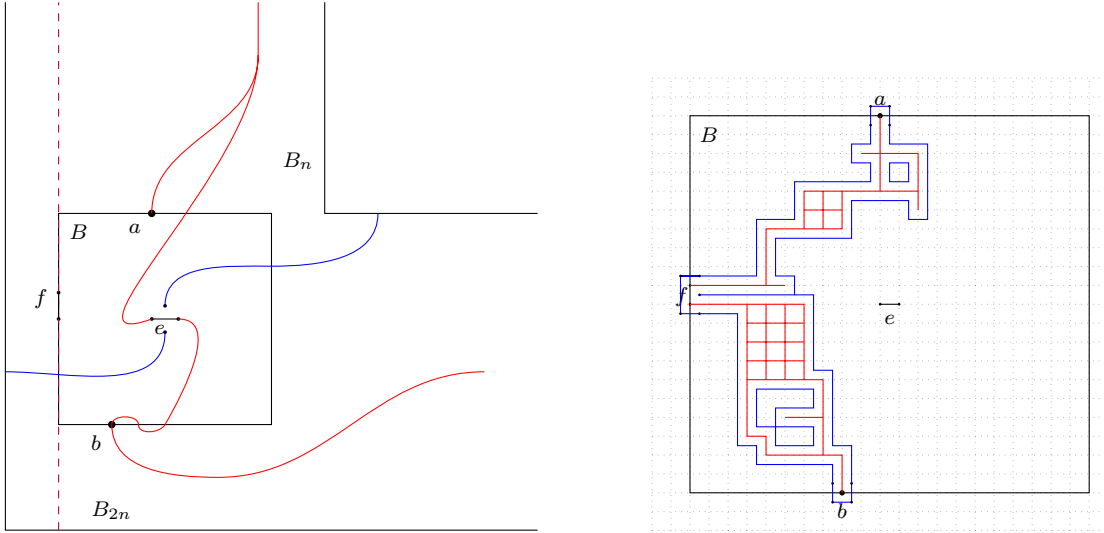
\begin{figure}
 \centering
  \begin{subfigure}[b]{0.5\textwidth}
\begin{tikzpicture}[scale=.7]
\draw (0,10) -- (0,0) -- (10,0);
\node[above] at (2,0) {\tiny $B_{2n}$};
\draw (6,10) -- (6,6) -- (10,6);
\node[left] at (6,7) {\tiny $B_n$};
\draw (1,2) rectangle (5,6);
\node[below right] at (1,6) {\tiny $B$};
\coordinate (a) at (2.75,6);
\coordinate (b) at (2,2);
\coordinate (x) at (2.75,4);
\coordinate (y) at (3.25,4);
\coordinate (z) at (1,4);
\coordinate (w) at (1,4.5);
\fill (x) circle [radius=1pt]
      (y) circle [radius=1pt]
      (a) circle [radius=2pt] node[below left] {\tiny $a$}
      (b) circle [radius=2pt] node[below left] {\tiny $b$}
      (z) circle[radius=1pt] node[above left] {\tiny $f$}
      (w) circle[radius=1pt]
      ($(x)+(.25,-.25)$) circle [radius=1pt]
      ($(x)+(.25,.25)$) circle [radius=1pt];
\draw (2.75,4)--(3.25,4);
\node[below] (e) at (2.9,4.1) {\tiny $e$};

\draw[red] (a) to [out=90, in=270] ($(a)+(2,3)$) to [out=90, in=270] ($(a)+(2,4)$)
          (x) to [out =200, in=270] ($(a)+(2,3)$)
          (b) to [out=270, in=180] ($(b)+(2,-1)$) to [out=0, in=180] ($(b)+(7,1)$)
          (b) to [out=90, in=90] ($(b)+(.5,0)$) to [out=270, in=240] ($(b)+(1,0)$) to [out=60, in=0] (y);

\draw[blue] (0,3) to [out=0, in=270] ($(x)+(.25,-.25)$)
	    ($(x)+(.25,.25)$) to [out=90, in=180] (5,5) to [out=0, in=270] (7,6);
	    
\draw[dashed, color=purple] (1,0) -- (1,10);
          
\draw (z) -- (w);
\end{tikzpicture}
\end{subfigure}
\quad
\begin{subfigure}[b]{0.4\textwidth}
\begin{tikzpicture}[scale=.5]

\draw[help lines, step=.5, dotted]  (-1,-1) grid (11,11);
\draw (0,0) rectangle (10.5,10);
\node at (.5,9.5) {\tiny $B$};
\draw (5,5) -- (5.5,5)
 node at (5.25,5) [below] {\tiny $e$};
\fill (5,5) circle [radius=1pt]
(5.5, 5) circle [radius=1pt];
\coordinate (x2) at (5,10); \node[above] at (x2) {\tiny $a$}; \fill (x2) circle [radius=2pt];
\coordinate (x1) at (5.25, 9.75);
\coordinate (x3) at (4.75, 9.75);
\coordinate (x7) at (0,5);
\fill (x7) circle [radius=1pt];
\coordinate (x5) at (0,5.5); 
\fill (x5) circle [radius=1pt];
\coordinate (x6) at (.25,5.25);
\coordinate (x4) at (.25, 5.75); \node[below left] at (x4) {\tiny $f$};
\coordinate (x8) at (.25, 4.75);
\coordinate (x10) at (4,0); \node[below] at (x10) {\tiny $b$}; \fill (x10) circle [radius=2pt];
\coordinate (x9) at (3.75, .25);
\coordinate (x11) at (4.25, .25);
\draw[blue] (x4) -- ($(x4)!0.5cm!315:(x5)$)
(x3) -- ($(x3)+(0,0.5)$)
(x9) -- ($(x9)+(0,-0.5)$)
(x8) -- ($(x8)+(-0.5,0)$)
(x11) -- ($(x11)+(0,-.5)$)
(x1) -- ($(x1)+(0,0.5)$)
(x4) -- ($(x4)+(-0.5,0)$) -- ($(x8)+(-0.5,0)$)
($(x1)+(0,.5)$) -- ($(x1)+(-.5,.5)$)
($(x11)+(0,-.5)$) -- ($(x11)+(-.5,-.5)$);
\fill (x4) circle [radius=1pt]
(x1) circle [radius=1pt]
($(x1)+(0,0.5)$) circle [radius=1pt]
($(x4)+(-0.5,0)$) circle [radius=1pt]
(x3) circle [radius=1pt]
($(x3)+(0,0.5)$) circle [radius=1pt]
(x8) circle [radius=1pt]
($(x8)+(-0.5,0)$) circle [radius=1pt]
(x9) circle [radius=1pt]
($(x9)+ (0,-.5)$) circle [radius=1pt]
(x6) circle [radius=1pt]
(x11) circle [radius=1pt]
($(x11)+(0,-0.5)$) circle [radius=1pt];
\draw[red] (x2) -- ($(x2)+(0,-1)$) -- ($(x2)+(1,-1)$) -- ($(x2)+(1,-2)$) -- ($(x2)+(0,-2)$) -- ($(x2)+(-1,-2)$) -- ($(x2)+(-1,-3)$) -- ($(x2)+(-3,-3)$) -- ($(x2)+(-3,-4.5)$) -- ($(x2)+(-5,-4.5)$);
\draw[red] ($(x2)+(0,-1)$) -- ($(x2)+(-.5,-1)$) ($(x2)+(-1,-2)$) -- ($(x2)+(-2,-2)$) -- ($(x2)+(-2,-3)$) ($(x2)+(-1,-2.5)$) -- ($(x2)+(-2,-2.5)$) ($(x2)+(-1.5,-2)$) -- ($(x2)+(-1.5,-3)$) ($(x2)+(0,-1)$) -- ($(x2)+(0,-2)$) ($(x2)+(1,-2)$) -- ($(x2)+(1,-2.5)$) ($(x2)+(-3,-4.5)$) -- ($(x2)+(-2.5,-4.5)$);
\draw[red] (x10) -- ($(x10)+(0,1)$) --  ($(x10)+(-2,1)$) --  ($(x10)+(-2,1.5)$) --  ($(x10)+(-2.5,1.5)$) --  ($(x10)+(-2.5,3)$) --  ($(x10)+(-1,3)$) -- ($(x10)+(-1,5)$) -- ($(x10)+(-4,5)$);
\draw[red] ($(x10)+(-2.5,5)$) -- ($(x10)+(-2.5,3)$)  ($(x10)+(-2,5)$) -- ($(x10)+(-2,3)$) ($(x10)+(-1.5,5)$) -- ($(x10)+(-1.5,3)$) ($(x10)+(-2.5,4.5)$) -- ($(x10)+(-1,4.5)$) ($(x10)+(-2.5,4)$) -- ($(x10)+(-1,4)$) ($(x10)+(-2.5,3.5)$) -- ($(x10)+(-1,3.5)$) ($(x10)+(-1,3)$) -- ($(x10)+(-.5,3)$) -- ($(x10)+(-.5,1)$) ($(x10)+(-.5,2)$) -- ($(x10)+(-1.5,2)$);
\draw[blue] (x3) -- ($(x3)+(0,-.5)$) -- ($(x3)+(-.5,-.5)$) -- ($(x3)+(-.5,-1)$) -- ($(x3)+(0,-1)$) -- ($(x3)+(0,-1.5)$) -- ($(x3)+(-2,-1.5)$) -- ($(x3)+(-2,-2.5)$) -- ($(x3)+(-3,-2.5)$) -- ($(x3)+(-3,-4)$) -- ($(x3)+(-4.5,-4)$);
\draw[blue] (x9) -- ($(x9)+(0,.5)$) -- ($(x9)+(-2,.5)$) -- ($(x9)+(-2,1)$) -- ($(x9)+(-2.5,1)$) -- ($(x9)+(-2.5,4.5)$) -- ($(x9)+(-3.5,4.5)$);
\draw[blue] (x1) -- ($(x1)+(0,-.5)$) -- ($(x1)+(1,-.5)$) -- ($(x1)+(1,-2.5)$) -- ($(x1)+(.5,-2.5)$) -- ($(x1)+(.5,-2)$) -- ($(x1)+(-1,-2)$) -- ($(x1)+(-1,-3)$) -- ($(x1)+(-3,-3)$) -- ($(x1)+(-3,-4)$) -- ($(x1)+(-2.5,-4)$) -- ($(x1)+(-2.5,-4.5)$) -- ($(x1)+(-5,-4.5)$) 
(x11) -- ($(x11)+ (0,1)$) -- ($(x11)+ (-.5,1)$) -- ($(x11)+ (-.5,3)$) -- ($(x11)+ (-1,3)$) -- ($(x11)+ (-1,5)$) -- ($(x11)+ (-1.5,5)$);
\draw[blue] ($(x2)+(.25,-1.25)$) rectangle  ($(x2)+(.75,-1.75)$);
\draw[blue] ($(x9) +(-.5,1)$) -- ($(x9) +(-1.5,1)$) -- ($(x9) +(-1.5,1.5)$) -- ($(x9) +(-2,1.5)$) -- ($(x9) +(-2,2.5)$) -- ($(x9) +(-.5,2.5)$) -- ($(x9) +(-.5,2)$) -- ($(x9) +(-1.5,2)$) -- ($(x9) +(-1.5,1.5)$) -- ($(x9) +(-.5,1.5)$) -- ($(x9) +(-.5,1)$) ;
\end{tikzpicture}
\end{subfigure}
\caption{(In red, primal edges, in blue dual edges) On the left we depict part the event $\mathcal{P}_{a,b}$. 
To have the complete event, the red paths starting from $a$ and $b$ have to remain inside $B_{2n}\setminus B_n$ and meet at some point. 
On the right we depict the event $\mathcal{G}_{a,b}$. 
The set $C_{a,b}$ comprises all the sites that are endpoints of the red edges.}
\label{fig:eventPab}
\end{figure}

Let $C_{a,b}=C_{a,b}(\omega)$ be the union of the clusters of $a$ and $b$ in $B$ for the configuration $\omega$, that is, the set of all sites that are connected to $a$ or to $b$ in $B$ for the configuration $\omega$ (see Figure \ref{fig:eventPab}).
Also denote
$$E_{a,b}=\big\{\{u,v\}\in E(\mathbb Z^2):u=a\text{ or }b\text{, and }v\notin B\big\}.$$ 
For a pair $(\omega, \xi) \in \{0,1\}^{E(\mathbb Z^2)} \times \{0,1\}^{E(\mathbb Z^2)}$, let $\Phi (\omega, \xi)$ be defined as follows: For $e'\in E(\mathbb Z^2)$, set
\begin{equation}
\Phi (\omega, \xi)(e'):=
\begin{cases}
0 & \text{if } e'=e \text{ and } e \not\subset {C}_{a,b}(\xi), \\
\xi(e') & \text{if } e'\notin E_{a,b} \text{ and $e'$ has at least one endpoint in } C_{a,b}(\xi), \\
\omega(e') & \text{otherwise}.
\end{cases}
\end{equation}
In the above, by $e \not\subset C_{a,b}(\xi)$ we mean that at least one of the endpoints of $e$ does not belong to $C_{a,b}$.
Roughly speaking, for getting $\Phi(\omega, \xi)$ we must ``superpose'' the edges in $C_{a,b}(\xi)$ (painted in red on the right side of Figure \ref{fig:eventPab}) together with the dual edges in its immediate neighborhood (painted in blue) on the configuration $\omega$.

We now claim that $\Phi (\omega, \xi) \in \{ f \text{ is pivotal for }
\mathcal A_n\}$ for $(\omega, \xi) \in \mathcal{P}_{a,b} \times
\mathcal{G}_{a,b}$. 
To see this, first note that when $f$ is open in $\xi$, $\Phi(\omega,\xi)$ must contain an open circuit in $B_{2n-1}$ that surrounds $B_n$. 
In fact this circuit can be taken as the union of the connection between $a$ and $b$ outside $B$ from $\omega$ and the connection inside $B$ from $\xi$. 
To see that, if $f$ is closed, there is no such open circuit note that, because $e$ is pivotal for $\mathcal A_n$ in $\omega$, after $e$ is closed, $\omega$ no longer contains an open path surrounding $B_n$. 
Superimposing $C_{a,b}$ from $\xi$ over $\omega$ does
not change this because $C_{a,b}$ comes with all the closed edges that surround it apart from the ones in $E_{a,b}$.
So the only open path it could possibly add inside $B$ is from $a$ to
$b$.
However, recalling that $f$ is pivotal for $\{a\lr[B]b\}$ for $\xi$, there can be no such path when $f$ is closed.
As a consequence,
\begin{equation}\label{eq:12}
\mathbb{P}_{p,q}^{\Lambda} (f \text{ is pivotal for } \mathcal A_n) \geq \mathbb{P}_{p,q}^{\Lambda} (\Phi(\mathcal{P}_{a,b} \times \mathcal{G}_{a,b})) \geq (1-p)\mathbb{P}_{p,q}^{\Lambda}(\mathcal{P}_{a,b})\mathbb{P}_{p,q}^{\Lambda}(\mathcal{G}_{a,b}) ,
\end{equation}
where $1-p$ accounts for the eventual price for closing $e$ when necessary, and the inequality is due to the fact that the law of $\Phi (\omega, \xi)$ in $E(\mathbb{Z}^2) \setminus \{e\}$ coincides with $\mathbb{P}_{p,q}^{\Lambda}$ (since $C_{a,b}(\xi)=C$ is measurable in terms of the states of edges with one endpoint in $C$).

On the one hand, if $e$ is pivotal for $\mathcal A_n$, one of the $\mathcal P_{a,b}$ must occur, allowing us to choose $a$ and $b$ so that 
$$\mathbb{P}_{p,q}^{\Lambda}(\mathcal P_{a,b})\ge \frac{1}{(8k)^2}\mathbb P_{p,q}^\Lambda(e\text{ pivotal for }\mathcal A_n).$$ 
(we use here that $|\partial B|\le 8k$). On the other hand, Lemma~\ref{l:surgery} below implies that
$\mathbb P_{p,q}^\Lambda(\mathcal G_{a,b})\ge k^{-c_6}.$
Putting these two inequalities in \eqref{eq:12} implies \eqref{eq:ll} for $k$ large enough. This concludes the proof.\end{proof}
In order to have a complete proof of Proposition~\ref{l:russo}, we only need to prove the following lemma. Recall the definition of $\mathcal{G}_{a,b}$ in \eqref{e:Aab}.
\begin{lemma}
\label{l:surgery}
There exists $c_6>0$ such that for any $k\ge 2$, if $\Lambda \times \mathbb{Z}$ does not intersect $B_\ell(z)\setminus \partial B_\ell(z)$ with $\ell\le k$ and $z\in\mathbb Z^2$, then
$$\mathbb P_{p,q}^\Lambda(\mathcal G_{a,b})\ge k^{-c_6}$$
for any $a,b\in\partial B_\ell(z)$, any $f \in E_{\text{vert}}(\Lambda \times \mathbb{Z}) \cap \partial{B_\ell(z)}$ and any $(p,q) \in [p_c, p_c+\varepsilon] \times [p_c-k^{-2}, p_c+k^{-2}]$. 
 \end{lemma}

Some readers may want to skip the proof of this lemma since it relies on very standard arguments involving the Russo-Seymour-Welsh theory at criticality \cite{Russo78,Seymour_Welsh} (see also \cite[Section 11.7]{Grimmett} for a comprehensive exposition).
For completeness, we include a proof here. 

\begin{proof}[Proof of Lemma~\ref{l:surgery}]
For simplicity, assume that $k\geq 200$ is divisible by $50$. Assume that $B:=\llbracket -k,k\rrbracket^2$. Also consider $\widehat B=\llbracket -k+100,k-100\rrbracket^2$ and $\widetilde B = \llbracket - k/2,k/2\rrbracket^2$. 
Set $f = \{c,d\}$, so that $a$, $b$, $c$ and $d$ are all lying on $\partial{B}$.
Assume that $a$, $b$, $c$ and $d$ are distinct vertices (the proof is similar in the other cases). 
\bigbreak
First, observe that $E_{\rm vert}(\Lambda\times\mathbb Z)$ does not intersect $B\setminus \partial B$. Therefore, the fact that $|q-p_c|\le k^{-2}$ and that $B$ has at most $ck^2$ edges, guarantees that there exists $c_7>0$ (independent of $k$) such that
for any configuration $\omega$ in $B\setminus\partial B$, 
$\mathbb P_{p,q}^\Lambda(\omega)\leq {c_7}\mathbb P_{p_c}(\omega)$.
So that we may focus on $p=q=p_c$.
\bigbreak
Partition each of the sides of $\partial{\widetilde B}$ into $25$ intervals of length $k/{50}$.
From this collection of intervals, select $11$ intervals $I_{1}, \ldots I_{11}$ arranged in increasing index order counter-clockwise along $\partial\widetilde B$ with the following properties:
\begin{itemize}
\item intervals are distant of $k/50$ from each other and from the corners of $\widetilde{B}$,
\item the intervals $I_1$, $I_2$ and $I_3$ are on the top (respectively left, right, bottom) side of $\partial \widetilde B$ if $a$ is on the top (respectively left, right, bottom) side of $\partial B$,
\item the intervals $I_4,\dots,I_8$ are on the left (respectively right) side of $\partial \widetilde B$ if $c$ and $d$ are on the left (respectively right) side of $\partial B$,
\item the intervals $I_9$, $I_{10}$ and $I_{11}$ are on the top (respectively left, right, bottom) side of $\partial \widetilde B$ if $b$ is on the top (respectively left, right, bottom) side of $\partial B$.
\end{itemize}
Define $\overline C_1$, $\overline C_2$ and $\overline C_3$ to be the cones from $a$ with basis $I_1$, $I_2$ and $I_3$. Similarly, define $\overline C_4,\dots,\overline C_8$ from $c$ with basis $I_4,\dots,I_8$ and $\overline C_9$, $\overline C_{10}$ and $\overline C_{11}$ from $b$ to $I_9$, $I_{10}$ and $I_{11}$. We set $C_i=\overline C_i\cap \widehat B\setminus (\widetilde B\setminus \partial \widetilde B)$ and $C_i^*=\overline C_i\cap\widehat B^*\setminus (\widetilde B^*\setminus \partial \widetilde B^*)$ (see Figure \ref{fig:cones}).

For $i = 1, \dots ,11$, select $x_i$ and $x_i^*$ in $C_i\cap\partial \widehat B$ and $C_i^*\cap\partial \widehat B^*$ respectively. Also select $z_i$ and $z_i^*$ in $C_i \cap \partial \widetilde B$ and $C_i^* \cap \partial\widetilde B^*$ respectively. Define
$$\mathcal G_i=\{x_i\lr[C_i]z_i\} \text{ and } \mathcal G_i^*=\{x_i^*\lr[*,C_i^*]z_i^*\}.$$

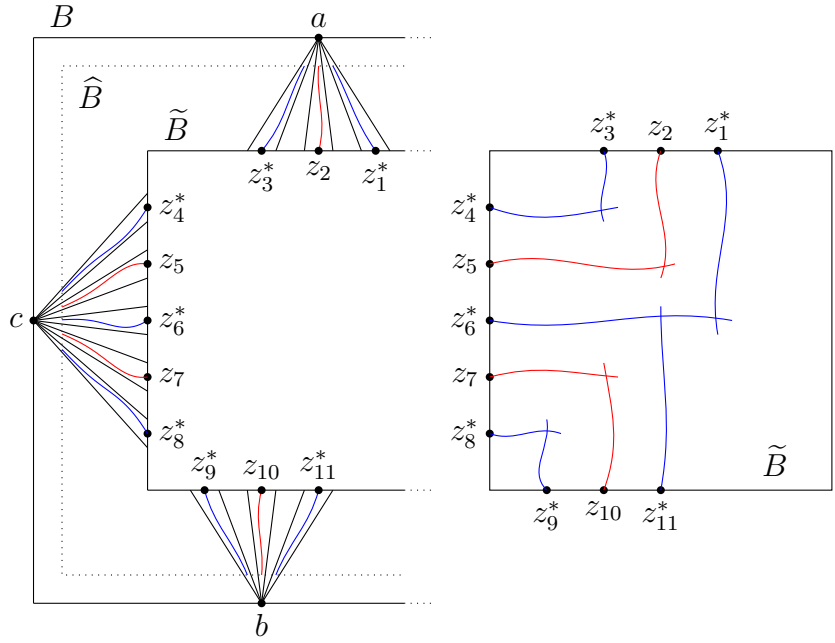
\begin{figure}
\begin{tikzpicture}[scale=.75]


\draw (0, 0) -- (0 , 10);
\draw (0,10) -- (6.5, 10);
\draw[dotted] (6.5, 10) -- (7, 10);
\draw (0,0) -- (6.5,0);
\draw[dotted] (6.5, 0) -- (7, 0);
\node[left, above] at (.5,10) {$B$};

\draw (2, 2) -- (2, 8);
\draw (2, 2) -- (6.5, 2);
\draw[dotted] (6.5, 2) -- (7, 2);
\draw (2, 8) -- (6.5, 8);
\draw[dotted] (6.5, 8) -- (7, 8);
\node[right, above] at (2.5,8) {$\widetilde{B}$};

\draw[dotted] (6.5,9.5) -- (.5,9.5) -- (.5,.5) -- (6.5,.5);
\node[right, below] at (1,9.5) {$\widehat{B}$};

\fill (0,5) circle[radius=2pt];
\node[left] at (0,5) {$c$};

\draw (0, 5) -- (2, 2.75);
\draw (0, 5) -- (2, 3.25);
\draw[blue, domain=0.5:2] plot (\x, {\x+ 5+ 0.05*\x^2*sin(\x*pi r)});
\fill (2,3) circle[radius=2pt];
\node[right] at (2,3) {$z_8^*$};

\draw (0, 5) -- (2, 3.75);
\draw (0, 5) -- (2, 4.25);
\draw[red, domain=0.5:2] plot (\x, {.5*\x+ 5 - 0.05*\x^2*sin(\x*pi r)});
\fill (2,4) circle[radius=2pt];
\node[right] at (2,4) {$z_7$};

\draw (0, 5) -- (2, 4.75);
\draw (0, 5) -- (2,  5.25);
\draw[blue, domain=.5:2] plot (\x, {5+ 0.05*\x^2*sin(\x*pi r)});
\fill (2,5) circle[radius=2pt];
\node[right] at (2,5) {$z_6^*$};

\draw (0, 5) -- (2, 5.75);
\draw (0, 5) -- (2, 6.25);
\draw[red, domain=.5:2] plot (\x, {-.5*\x + 5 + 0.05*\x^2*sin(\x*pi r)});
\fill (2,6) circle[radius=2pt];
\node[right] at (2,6) {$z_5$};

\draw (0, 5) -- (2, 6.75);
\draw (0, 5) -- (2, 7.25);
\draw[blue, domain=.5:2] plot (\x, {-\x + 5 - 0.05*\x^2*sin(\x*pi r)});
\fill (2,7) circle[radius=2pt];
\node[right] at (2,7) {$z_4^*$};

\fill (4,0) circle[radius=2pt];
\node[below] at (4,0) {$b$};

\draw (4, 0) -- (2.75, 2);
\draw (4, 0) -- (3.25, 2);
\draw[blue] (3.75,.5) to [out=112, in=285] (3,2);
\fill (3,2) circle[radius=2pt];
\node[above] at (3,2) {$z_9^*$};

\draw (4, 0) -- (3.75, 2);
\draw (4, 0) -- (4.25, 2);
\draw[red] (4, 0.5) to [out=85, in=255] (4, 2);
\fill (4,2) circle[radius=2pt];
\node[above] at (4,2) {$z_{10}$};

\draw (4, 0) -- (4.75, 2);
\draw (4, 0) -- (5.25, 2);
\draw[blue] (4.25,0.5) to [out=68, in=255] (5,2);
\fill (5,2) circle[radius=2pt];
\node[above] at (5,2) {$z_{11}^*$};

\fill (5,10) circle[radius=2pt];
\node[above] at (5,10) {$a$};

\draw (5, 10) -- (4.25, 8);
\draw (5, 10) -- (3.75, 8);
\draw[blue] (4.75,9.5) to [out=242, in=50] (4,8);
\fill (4,8) circle[radius=2pt];
\node[below] at (4,8) {$z_3^*$};

\draw (5, 10) -- (5.25, 8);
\draw (5, 10) -- (4.75, 8);
\draw[red] (5,9.5) to [out=265, in=75] (5,8);
\fill (5,8) circle[radius=2pt];
\node[below] at (5,8) {$z_2$};

\draw (5, 10) -- (6.25, 8);
\draw (5, 10) -- (5.75, 8);
\draw[blue] (5.25,9.5) to [out=298, in=125] (6,8);
\fill (6,8) circle[radius=2pt];
\node[below] at (6,8) {$z_1^*$};

\draw (8, 2) rectangle (14, 8);
\node[right, below] at (13, 3) {$\widetilde{B}$};

\fill (10,8) circle[radius=2pt];
\node[above] at (10,8) {$z_3^*$};
\draw[blue] (10,8) to [out=290, in=110] (10, 6.75);
\fill (11,8) circle[radius=2pt];
\node[above] at (11,8) {$z_2$};
\draw[red] (11,8) to [out=250, in=70] (11, 5.75);
\fill (12,8) circle[radius=2pt];
\node[above] at (12,8) {$z_1^*$};
\draw[blue] (12,8) to [out=290, in=100] (12, 4.75);

\fill (8,3) circle[radius=2pt];
\node[left] at (8,3) {$z_8^*$};
\draw[blue] (8,3) to [out=340, in=160] (9.25, 3);
\fill (8,4) circle[radius=2pt];
\node[left] at (8,4) {$z_7$};
\draw[red] (8,4) to [out=15, in=175] (10.25, 4);
\fill (8,5) circle[radius=2pt];
\node[left] at (8,5) {$z_6^*$};
\draw[blue] (8,5) to [out=350, in=170] (12.25, 5);
\fill (8,6) circle[radius=2pt];
\node[left] at (8,6) {$z_5$};
\draw[red] (8,6) to [out=15, in=195] (11.25, 6);
\fill (8,7) circle[radius=2pt];
\node[left] at (8,7) {$z_4^*$};
\draw[blue] (8,7) to [out=340, in=190] (10.25, 7);

\fill (9,2) circle[radius=2pt];
\node[below] at (9,2) {$z_9^*$};
\draw[blue] (9,2) to [out=135, in=280] (9, 3.25);
\fill (10,2) circle[radius=2pt];
\node[below] at (10,2) {$z_{10}$};
\draw[red] (10,2) to [out=70, in=280] (10, 4.25);
\fill (11,2) circle[radius=2pt];
\node[below] at (11,2) {$z_{11}^*$};
\draw[blue] (11,2) to [out=80, in=270] (11, 5.25);

\end{tikzpicture}
\caption{Dual paths are represented in blue and primal paths in red. On the left we show the events $\mathcal{G}_i$ (respectively\ $\mathcal{G}_i^*$) whose occurrence is assured by the existence of the red (respectively\ blue) paths inside the cones $C_i$  (respectively $C_i^*$).
On the right we show a simple way of constructing the event $\mathcal{H}$ once the $z_i^*$ and $z_i$ are well separated.}
\label{fig:cones}
\end{figure}

Also set 
$$\mathcal H=\big\{z_1^*\lr[*,\widetilde B^*] z_6^*, z_2\lr[\widetilde B] z_5, z_3^*\lr[*,\widetilde B^*] z_4^*,
 z_6^*\lr[*,\widetilde B^*] z_{11}^*,z_7\lr[\widetilde B] z_{10},z_8^*\lr[*,\widetilde B^*] z_9^*\big\}.$$
By a standard application of the Russo-Seymour-Welsh arguments, there exists $c_8 >0$ (not depending on the choice of the $I_i$, $x_i$, $z_i$, etc.) such that for $k$ large enough
\[
\mathbb{P}_{p_c}(\mathcal{G}_i) \geq k^{-c_8}, \quad \mathbb{P}_{p_c}(\mathcal{G}^*_i) \geq k^{-c_8}.
\]
Now, the sites $z_i$ and $z_i^*$ are all well separated so that one may again employ Russo-Seymour-Welsh arguments in order to check that
$$\mathbb P_{p_c}(\mathcal H)\ge k^{-c_8}$$
(where the value constant $c_8$ may need to be modified). 

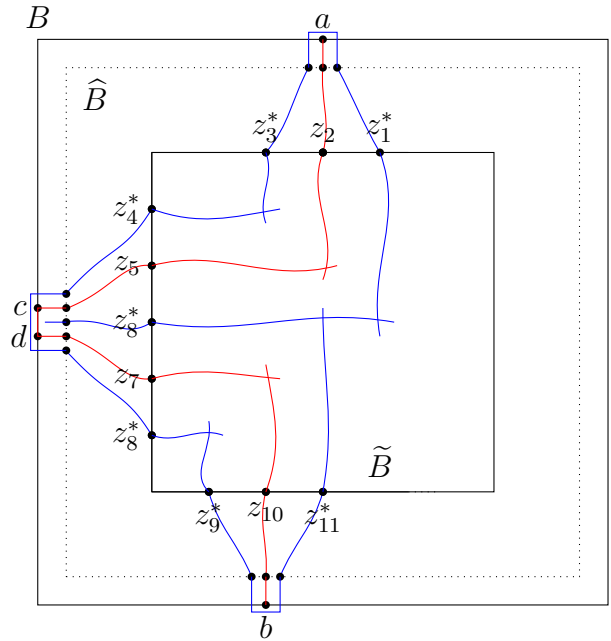
\begin{figure}
\begin{tikzpicture}[scale=.75]


\draw (0,0) rectangle (10,10);
\node[right, above] at (0,10) {$B$};

\draw (2, 2) -- (2, 8);
\draw (2, 2) -- (6.5, 2);
\draw[dotted] (6.5, 2) -- (7, 2);
\draw (2, 8) -- (6.5, 8);
\node[right, below] at (6,3) {$\widetilde{B}$};

\draw[dotted] (0.5,.5)--(.5,9.5)--(9.5,9.5)--(9.5,.5)--cycle;
\node[left, below] at (1,9.5) {$\widehat{B}$};

\fill (0,5.25) circle[radius=2pt];
\fill (0,4.75) circle[radius=2pt];
\node[left] at (0,5.25) {$c$};
\node[left] at (0,4.75) {$d$};

\draw[blue, domain=.5:2] plot (\x, {\x+ 5+ 0.05*\x^2*sin(\x*pi r)});
\fill (2,3) circle[radius=2pt];

\draw[red, domain=.5:2] plot (\x, {.5*\x+ 5 - 0.05*\x^2*sin(\x*pi r)});
\fill (2,4) circle[radius=2pt];

\draw[blue, domain=.5:2] plot (\x, {5+ 0.05*\x^2*sin(\x*pi r)});
\fill (2,5) circle[radius=2pt];

\draw[red, domain=.5:2] plot (\x, {-.5*\x + 5 + 0.05*\x^2*sin(\x*pi r)});
\fill (2,6) circle[radius=2pt];

\draw[blue, domain=.5:2] plot (\x, {-\x + 5 - 0.05*\x^2*sin(\x*pi r)});
\fill (2,7) circle[radius=2pt];

\fill (.5,5) circle[radius=2pt];
\fill (.5,4.75) circle[radius=2pt];
\fill (.5,4.5) circle[radius=2pt];
\fill (.5,5.25) circle[radius=2pt];
\fill (.5,5.5) circle[radius=2pt];

\fill (4,0) circle[radius=2pt];
\node[below] at (4,0) {$b$};

\draw[blue] (3.75,0.5) to [out=112, in=285] (3,2);
\fill (3,2) circle[radius=2pt];
\fill (3.75,0.5) circle[radius=2pt];

\draw[red] (4, 0.5) to [out=85, in=255] (4, 2);
\fill (4,2) circle[radius=2pt];
\fill(4, 0.5) circle[radius=2pt];

\draw[blue] (4.25,0.5) to [out=68, in=255] (5,2);
\fill (5,2) circle[radius=2pt];
\fill(4.25,0.5) circle[radius=2pt];

\fill (5,10) circle[radius=2pt];
\node[above] at (5,10) {$a$};

\draw[blue] (4.75,9.5) to [out=242, in=50] (4,8);
\fill (4,8) circle[radius=2pt];
\fill (4.75,9.5) circle[radius=2pt];

\draw[red] (5,9.5) to [out=265, in=75] (5,8);
\fill (5,8) circle[radius=2pt];
\fill (5,9.5) circle[radius=2pt];

\draw[blue] (5.25,9.5) to [out=298, in=120] (6,8);
\fill (6,8) circle[radius=2pt];
\fill (5.25, 9.5) circle[radius=2pt];

\draw (2, 2) rectangle (8, 8);

\fill (4,8) circle[radius=2pt];
\node[above] at (4,8) {$z_3^*$};
\draw[blue] (4,8) to [out=290, in=110] (4, 6.75);
\fill (5,8) circle[radius=2pt];
\node[above] at (5,8) {$z_2$};
\draw[red] (5,8) to [out=250, in=70] (5, 5.75);
\fill (5,8) circle[radius=2pt];
\node[above] at (6,8) {$z_1^*$};
\draw[blue] (6,8) to [out=290, in=100] (6, 4.75);

\fill (2,3) circle[radius=2pt];
\node[left] at (2,3) {$z_8^*$};
\draw[blue] (2,3) to [out=340, in=160] (3.25, 3);
\fill (2,4) circle[radius=2pt];
\node[left] at (2,4) {$z_7$};
\draw[red] (2,4) to [out=15, in=175] (4.25, 4);
\fill (2,5) circle[radius=2pt];
\node[left] at (2,5) {$z_8^*$};
\draw[blue] (2,5) to [out=350, in=170] (6.25, 5);
\fill (2,6) circle[radius=2pt];
\node[left] at (2,6) {$z_5$};
\draw[red] (2,6) to [out=15, in=195] (5.25, 6);
\fill (2,7) circle[radius=2pt];
\node[left] at (2,7) {$z_4^*$};
\draw[blue] (2,7) to [out=340, in=190] (4.25, 7);

\fill (3,2) circle[radius=2pt];
\node[below] at (3,2) {$z_9^*$};
\draw[blue] (3,2) to [out=135, in=280] (3, 3.25);
\fill (4,2) circle[radius=2pt];
\node[below] at (4,2) {$z_{10}$};
\draw[red] (4,2) to [out=70, in=280] (4, 4.25);
\fill (5,2) circle[radius=2pt];
\node[below] at (5,2) {$z_{11}^*$};
\draw[blue] (5,2) to [out=80, in=270] (5, 5.25);

\draw[blue] (.5,5.5) -- (-.125,5.5) -- (-.125,4.5) -- (.5,4.5);
\draw[red] (.5,5.25) -- (0,5.25) -- (0,4.75) -- (.5,4.75);
\draw [blue] (.5,5) -- (.125,5);
\draw[blue] (3.75,.5) -- (3.75,-.125) -- (4.25,-.125) -- (4.25,.5);
\draw[red] (4,.5) -- (4,0);
\draw[blue] (4.75, 9.5) --(4.75,10.125) -- (5.25,10.125) -- (5.25,9.5);
\draw[red] (5, 9.5) -- (5,10);

\end{tikzpicture}
\caption{The event $\mathcal{H}$, together with the appropriate $\mathcal{G}_i$'s and $\mathcal{G}_i^*$'s and a local surgery in the neighbourhoods of $a$, $b$, $c$ and $d$ implies the occurrence of the event $\mathcal{P}_{a,b}$.}
\end{figure}

We deduce that
$$\mathbb P_{p_c}(\mathcal H\cap \mathcal G_1^*\cap\mathcal G_2\cap\mathcal G_3^*\cap\mathcal G_4^*\cap\mathcal G_5\cap \mathcal G_6^*\cap\mathcal G_7\cap\mathcal G_8^*\cap \mathcal G_9^*\cap\mathcal G_{10}\cap\mathcal G_{11}^*)\ge k^{-12 c_8}.$$
One concludes the proof by noticing that a local surgery near $a$, $b$, $c$ and $d$ implies the existence of $c_9>0$ such that
$$\mathbb P_{p_c}(\mathcal{P}_{a,b})\ge c_9\mathbb P_{p_c}(\mathcal H\cap \mathcal G_1^*\cap\mathcal G_2\cap\mathcal G_3^*\cap\mathcal G_4^*\cap\mathcal G_5\cap \mathcal G_6^*\cap\mathcal G_7\cap\mathcal G_8^*\cap \mathcal G_9^*\cap\mathcal G_{10}\cap\mathcal G_{11}^*)\ge c_9k^{-12 c_8}.$$
The proof follows by choosing $c_6>0$ large enough.
\end{proof}

\begin{remark*}Under the conjecture that $2$-dimensional percolation is
conformally invariant, the best constant in Lemma \ref{l:surgery}
may be calculated. The worst case is when $a$ and $b$ are both in
the same corner, and then one gets a half-plane 5-arm exponent at $f$
and a quarter-plane 5-arm exponent at the corner (which is twice the
half-plane exponent, by conformal invariance). Using the determination
of these exponents in \cite{Smirnov_Werner} gives $c_6=15$.
\end{remark*}

\section{Input from near-critical percolation}\label{sec:3}

In this section, we recall general facts on planar Bernoulli
percolation which imply Proposition~\ref{l:limpp_c} (recall that it
claimed that $\mathbb P_{p_c+(\log n)^{-c}}(\mathcal A_n)\to 1$ as $n \to \infty$). Let $1\gg \ep>0$. For $p>p_c$, introduce
$$L_\ep(p):=\min\Big\{n\geq 1;\; \mathbb{P}_p\Big(\{0\}\times\llbracket 0,n\rrbracket\lr[\llbracket 0,2n\rrbracket\times\llbracket 1,n-1\rrbracket]\{2n\}\times\llbracket 0,n\rrbracket\Big) \geq 1-\ep\Big\}.$$
This quantity, sometimes called {\em characteristic} or {\em correlation} length, was proved \cite{Kesten87, Nolin} to satisfy the following facts:
\begin{itemize}
\item[{\bf P1}] (Probability for hard-way crossings). For any $p>p_c$ and $n\ge L_\ep(p)$,
\begin{equation}\label{eq:P1}
\mathbb{P}_p\Big(\{0\}\times\llbracket 0,n\rrbracket
  \lr[\llbracket 0,2n\rrbracket\times\llbracket 1,n-1\rrbracket]
  \{2n\}\times\llbracket 0,n\rrbracket\Big) 
\geq 1-\ep.
\end{equation}
\item[{\bf P2}] (Probability for 4 arms). There exist $c_{10},c_{11}>0$ such that for any $p\in (p_c,1-\ep)$,
\begin{equation}\label{eq:P2}
c_{10}\le (p-p_c) L_\ep(p)^2\mathbb P_{p_c}\big(\mathcal
E_4(L_\ep(p))\big)\le c_{11},
\end{equation}
where (below, $x$ is a fixed neighbour of the origin)
$$\mathcal E_4(n):=\{0\lr[]\partial B_n\}\cap\{x\lr[]\partial B_n\}\cap\{0\lr[B_n]x\}^c.$$
\end{itemize}

Let us recall the following fact, of which we provide a sketch of proof for completeness.
\begin{lemma}\label{lem:bound on pi_4}
There exists $c_{12}<2$ such that for any $n$ large enough,
$\mathbb P_{p_c}(\mathcal E_4(n))\ge n^{-c_{12}}$.
\end{lemma}
\begin{proof}[Sketch of proof]
Let 
$$\mathcal E_5(n):=\mathcal E_4(n)\cap\{\text{there exist two open paths from $0$ to $\partial B_n$ intersecting at $0$ only}\}.$$ It was proved in \cite[Lemma 5]{Kesten_Sidoravicius_Zhang} that there exists $c_{13}>0$ such that
\begin{equation}
\label{e:PpcE5}
\mathbb P_{p_c}[\mathcal E_5(n)]\ge \frac{c_{13}}{n^2}.
\end{equation}
Since the occurrence of $\mathcal E_5(n)$ implies the disjoint occurrence (see \cite[Section 2.3]{Grimmett} for a definition of disjoint occurrence) of $\mathcal E_4(n)$ and $\{0\lr\partial B_n\}$, Reimer's inequality \cite{Reimer} implies that
\begin{equation}\label{eq:142}\mathbb P_{p_c}(\mathcal E_5(n))\le \mathbb P_{p_c}(\mathcal E_4(n))\cdot \mathbb P_{p_c}(0\lr \partial B_n).\end{equation}
Now, a simple application of the Russo-Seymour-Welsh theory \cite{Russo78, Seymour_Welsh} implies that 
$$\mathbb P_{p_c}(0\lr \partial B_n) \leq n^{-c_{14}}$$ for all $n\ge1$. Plugging this estimate in \eqref{eq:142} and using \eqref{e:PpcE5} implies the claim readily.\end{proof}

\begin{proof}[Proof of Proposition~\ref{l:limpp_c}] Fix $c_3>0$ and
  $\ep>0$. Lemma~\ref{lem:bound on pi_4} and \eqref{eq:P2} show that 
$$L_\ep(p)\le \left(\frac{c_{11}}{p-p_c}\right)^{1/(2-c_{12})}.$$
Thus, there exists $n_0=n_0(\ep)>0$ such that for all $n>n_0$,
$$L_\ep\big(p_c+(\log n)^{-c_3}\big)\le \left(c_{15}{\log n}\right)^{c_3/(2-c_{12})}\le n.$$
As a consequence, \eqref{eq:P1} implies that for $n\ge n_0$,
\begin{equation}\label{eq:112}\mathbb{P}_{p_c + \log{n}^{-c_{3}}}\Big(\{0\}\times\llbracket 0,n\rrbracket\lr[\llbracket 0,2n\rrbracket\times\llbracket 1,n-1\rrbracket]\{2n\}\times\llbracket 0,n\rrbracket\Big) \geq 1-\ep.\end{equation}
Now, assume that the following events occur simultaneously for $i,j\in\{-2,-1,0,1\}$,
\begin{itemize}
\item $\{in\}\times\llbracket jn,(j+1)n\rrbracket\lr[\llbracket in,(i+2)n\rrbracket\times\llbracket jn+1,(j+1)n-1\rrbracket]\{(i+2)n\}\times\llbracket jn,(j+1)n\rrbracket$,
\item $\llbracket in,(i+1)n\rrbracket\times\{jn\}\lr[\llbracket in+1,(i+1)n-1\rrbracket\times\llbracket jn,(j+2)n\rrbracket] \llbracket in,(i+1)n\rrbracket\times\{(j+2)n\}$.
\end{itemize}
In such case $\mathcal A_n$ occurs. 
(Note that we have been wasteful in the number of events involved above, some of them are not necessary in order to guarantee the occurrence of $\mathcal{A}_n$.) 
Therefore, the FKG inequality combined with \eqref{eq:112} implies that for $n\ge n_0$,
$$\mathbb P_{p_c + \log{n}^{-c_3}} (\mathcal A_n)\ge (1-\ep)^{32}$$
which implies the claim readily.
\end{proof}

\section{The renormalisation scheme}\label{sec:4}

\begin{figure} 
\begin{tikzpicture}[scale=.75]
\foreach \x in {0,...,14} 
  \foreach \y in {0,...,6}
  {
    \pgfmathparse{int(mod(\x+\y,2))}
    \let\r\pgfmathresult
    \ifnum\r=0
      \draw[fill=white] (\x,\y) rectangle (\x+1,\y+1);
      \fill (\x+.5,\y+.5) circle [radius=2pt];
      \draw[dotted] (\x,\y) -- (\x+1,\y+1);
      \draw[dotted] (\x+1, \y) -- (\x, \y+1);
    \else
      \draw[fill=gray!20] (\x,\y) rectangle (\x+1,\y+1);
    \fi
  }

\draw (0,-1) -- (15,-1);
\foreach \x in {-7,...,7}
 { 
   \draw[decoration ={brace, mirror},
   decorate] (7.1+\x,-.1) -- (7.9+\x, -.1);
   \node[below] (column \x) at (7+\x+.5,-0.1) {\tiny $c{(\x)}$};
   \draw (7+\x,-.9)--(7+\x,-1.1)
   (7.5+\x,-.9)--(7.5+\x,-1.1);
  
  }
  \draw (15,-.9)--(15,-1.1);
  \node[below] at (7.5,-1.1) {\tiny $0$};
\foreach \z in {-15,-13,...,-3,3,5,...,15}
 { \node[below] at (7.5+\z*.5, -1.1) {\tiny $\z n$}; }
 \node[below] at (8,-1.17) {\tiny $n$};
 \node[below] at (6.93,-1.1) {\tiny $-n$};
  
\draw (-1,0) -- (-1,7);  
\node[left] at (-1.1,3.5) {\tiny $0$};
\foreach \y in {0,...,6}
 {
 \draw (-1.1,\y) -- (-.9, \y)
 (-1.1, \y+.5) -- (-.9,\y+.5); 
 }
\draw (-1.1, 7) -- (-.9,7);
\foreach \z in {-7,-5,-3,3,5,7}
 {
 \node[left] at (-1.1,3.5+ \z*.5) {\tiny $\z n$};
 }
\node[left] at (-1.1,4) {\tiny $n$};
\node[left] at (-1.1,3) {\tiny $-n$};

\coordinate (origin) at (7.5, 3.5);
\fill (origin) circle [radius=1pt] node[below] {\tiny $0$};

\draw[red] 
          ($(origin)+(-.75,-.75)$) to [out=90+rand*15, in=270+rand*15] ($(origin)+(-.75,.75)$)
          ($(origin)+(.75, -.75)$) to [out=90+rand*15, in=270+rand*15] ($(origin)+(.75, .75)$)
          ($(origin)+(-.75,-.75)$) to [out=0+rand*15, in=180+rand*15] ($(origin)+(.75,-.75)$)
          ($(origin)+(-.75,.75)$) to [out=0+rand*15, in=180+rand*15] ($(origin)+(.75,.75)$);
          
\draw ($(origin)+(-1,-1)$) rectangle  ($(origin)+(1,1)$);

\foreach \x in {0,...,8}
 {
  \draw[dashed] (10.5+\x/4,0) -- (10.5+\x/4,7.5);
 } 

\foreach \x / \y in {0 / 2, 1 / 3, 2 / 2, 3 / 1, 4 / 2, 5 / 3 } 
 {
  \fill (\x+.5,\y+.5) circle [radius=1pt] node[above] {\tiny $i_{\x}$};
  \draw[red] 
          ($(\x+.5,\y+.5)+(-.75,-.75)$) to [out=90+rand*15, in=270+rand*15] ($(\x+.5,\y+.5)+(-.75,.75)$)
          ($(\x+.5,\y+.5)+(.75, -.75)$) to [out=90+rand*15, in=270+rand*15] ($(\x+.5,\y+.5)+(.75, .75)$)
          ($(\x+.5,\y+.5)+(-.75,-.75)$) to [out=0+rand*15, in=180+rand*15] ($(\x+.5,\y+.5)+(.75,-.75)$)
          ($(\x+.5,\y+.5)+(-.75,.75)$) to [out=0+rand*15, in=180+rand*15] ($(\x+.5,\y+.5)+(.75,.75)$);
 }
\end{tikzpicture}
\caption{In the center, we illustrate the event $\mathcal{A}_n (0)$ whose occurrence is assured by red circuit in the annulus $B_{2n-1} \setminus B_n$.
In the left, we show an oriented path of sites $i_0, \ldots, i_5$ around which the corresponding events occurs.
The dashed vertical lines around $c_4(n)$ cover the region where one needs to inspect $\Lambda \times \mathbb{Z}$ in order to verify that it is a good column.}
\label{fig:L}
\end{figure}
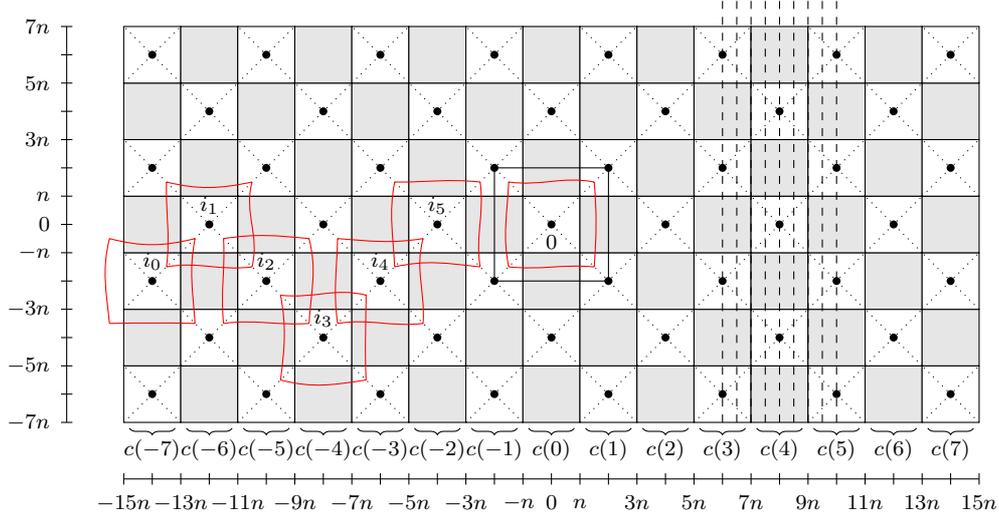

Recall that $\Lambda$ is a random subset of $\mathbb{Z}$ having law $\nu_{\rho}$ under which the events $\{ i\in\Lambda\}$ are mutually independent and have probability $\rho$.
Also recall that the
$i^\textrm{th}$ column of $\nesearrows$ (denoted by $c(i)$) is called good if $\Lambda$ intersects every subinterval of $\llbracket 2n(i-1),2n(i+1)\rrbracket$ that has diameter $\lceil \frac{2}{\rho}\log(2n)\rceil$. 
We start by proving that columns are good with high probability.

\begin{lemma}
\label{l:good}
Let $\rho >0$.
For every $i \in \mathbb{Z}$, $\displaystyle{\lim_{n \to \infty} \nu_{\rho} (c(i) \text{ is good}) =1}$.
\end{lemma}
\begin{proof}
For any $x \in \mathbb{Z}$ and $\Lambda\subset\mathbb Z$ let 
$$\ell(x)=\ell(x,\Lambda) := \inf\{w-x\, :\, w \in \Lambda,\, w> x\}.$$
A calculation gives
$\nu_{\rho}(\ell(x) > k) = (1-\rho)^k \leq e^{-\rho k}$.
It follows that
\[
\nu_{\rho} \big(\ell(x) > k \text{ for some } x \text{ between }2i(n-1)\text{ and } 2i(n+1)\big)\leq 4n e^{-\rho k}.
\]
For $k = \lceil \frac{2}{\rho}\log(2n) \rceil$, the right-hand side is at most equal to $1/n$ while the left-hand side contains the event that $c(i)$ is bad.
This proves the result.
\end{proof}

Let us dedicate a paragraph to the nature of the last remaining
obstacle. We already established that every column is good with high
probability; that in good columns $\mathbb{E}(X(v))$ can be
made as close to 1 as we wish; and that in bad columns 
$\mathbb{E}(X(v))>c$ (recall that $X(v)=\mathbf{1}[\mathcal A_n(2nv)]$). 
The problem is that the $X(v)$ at different $v$ are not independent, they are only
1-dependent. Now, the boxes in good columns do not pose any problem:
By \cite{Liggett_Schonmann_Stacey} 1-dependent
events with sufficiently high probability stochastically dominate
independent events with lower probability. It is the boxes in the bad
column that we must worry about. Liggett, Schonmann and Stacey give a
simple and highly instructive example of 1-dependent events with
probability $\nicefrac 12$ which do not dominate $p$-independent
events, no matter how small $p$ is taken (see the bottom of page 73 in
\cite{Liggett_Schonmann_Stacey}). Hence, if we want to show that the
collection of $X(v)$ dominates independent percolation, we
need to use some specific property of them. We will use the FKG inequality.

In the rest of this article, we will drop $n$ from the
notation. 
While it is probably
true that $\mathbb E(X(v)\,|\,X(w)\;\forall w\ne v)>c$, we found it easier
to use external randomization. Let us therefore introduce a
parameter $\eta\le\frac 12$ (to be fixed later) and a 
family of i.i.d.\ Bernoulli($1-\eta$) random variables
$(Y(z):z\in \nesearrows)$ which is independent of everything else. Define $W(z):=X(z)Y(z)$.
\begin{lemma}
\label{l:domW}
For any $\Lambda \subset \mathbb{Z}$ and $p$, $q \in [0,1]$, the following inequality holds almost surely
\[
\mathbb{P}_{p,q}^{\Lambda} (W(0) =1 \,|\, W(z), z\in \nesearrows\setminus\{0\}) \geq 
\eta^{4}\mathbb{P}_{p,q}^{\Lambda}(W(0) =1).
\]
\end{lemma}
\begin{proof}
For simplicity we will remove $\Lambda$, $p$ and $q$ from the notation.
Let $z_1,\dotsc,z_4$ be the 4 neighbours of 0 in $\nesearrows$ (since $\nesearrows$ is directed, we should specify that we mean either $0\sim z_i$ or $z_i\sim 0$). Let
$\zeta\in\{0,1\}^4$, let $m>2n$ and
let $\xi\in \{0,1\}^{B_m\setminus B_{2n}}$. We define the event 
\[
\mathcal{B} =\mathcal{B}_{\zeta,\xi}=
\{W(z_i)=\zeta_i\;\forall i\in\llbracket 1,4\rrbracket\}\cap
\{\omega(e)=\xi(e)\;\forall e\in E(B_m\setminus B_{2n})\}
\]
where $\omega$ is as before the percolation configuration. It is
enough to show
\begin{equation}\label{eq:Eenough}
\mathbb P(W(0)=1\,|\,\mathcal{B})
\ge\eta^{4}\mathbb P(W(0)=1)\qquad \text{a.s.}\qquad\forall\,\zeta,\xi.
\end{equation}
Indeed, once \eqref{eq:Eenough} is shown, it
is possible to add an arbitrary conditioning on
$\{Y(v)\,|\,v\not\in\{0,z_1,\dotsc,z_4\}\}$ as these are independent of
both $W(0)$ and of $\mathcal{B}$. Then taking $m\to\infty$ would give the
statement of the lemma but on a finer $\sigma$-field. Integrating
would give the exact claim of the lemma. 

We thus focus on the proof of \eqref{eq:Eenough}. Fix some $\zeta$, $m$ and
$\xi$ for the rest of the proof. Define
\[
\mu(\cdot)=\mathbb P(\,\cdot\,|\;\omega(e)=\xi(e)\;\forall e\in
E(B_m\setminus B_{2n})).
\]
Define $I(\zeta)=\{i:\zeta_i=1\}$ and then write
\[
\begin{split}
& \mu(X(0)=1, W(z_i)=\zeta_i\;\forall i\in\llbracket 1,4\rrbracket)\\
& \geq  \mu(X(0)=1, W(z_i)=\zeta_i, Y(z_i) = \zeta_i \;\forall
  i\in\llbracket 1,4\rrbracket) \\ 
& = \mu(X(0)=1, X(z_i)=1 \;\forall i \in I(\zeta),
  Y(z_i)=\zeta_i \;\forall i\in\llbracket 1,4\rrbracket)\\
& = \mu(X(0) = 1, X(z_i) =1 \;\forall i \in I(\zeta)) \cdot\mathbb P(Y(z_i)=\zeta_i \;\forall i\in\llbracket 1,4\rrbracket) \\
&\ge \eta^{4} \mu(X(0)=1, X(z_i)=1 \;\forall i\in I(\zeta))\\
& \geq \eta^{4} \mu(X(0)=1, W(z_i)=1 \;\forall i \in I(\zeta)).
\end{split}
\]
Noting that
$$\mu(W(z_i) =1 \;\forall i\in I(\zeta))
\ge \mu (W(z_i) = \zeta_i \;\forall i\in\llbracket 1,4\rrbracket) ,$$
one can conclude that
\begin{align*}
\lefteqn{\mu(X(0)=1|W(z_i)=\zeta_i \;\forall i\in\llbracket 1,4\rrbracket)}\qquad &\\
&\geq \eta^{4} \mu(X(0)=1| W(z_i)=1 \;\forall i \in I(\zeta))\\
\textrm{by FKG}\qquad & \geq \eta^{4} \mu(X(0)=1)
=\eta^{4}\mathbb P(X(0)=1).
\end{align*}
(It is easy to check that the FKG inequality holds for $\mu$).
Since $Y(0)$ is independent of $X(0)$ and $W(z_i)$ for $i=1,\ldots,4$
(and also after the conditioning on $E$), we have that
\[ \mu(W(0)=1|W(z_i)=\zeta_i \;\forall i\in\llbracket 1,4\rrbracket)
\geq \eta^{4}\mathbb P(W(0)=1).\]
which is equivalent to \eqref{eq:Eenough}, proving the lemma.
\end{proof}

We are now ready for: 

\begin{proof}[Proof of Theorem \ref{t:main}]As parameter dependency is
a little complicated here, let us start by setting all
parameters formally. First choose $\eta$ (from the definition of $Y$ and $W$)
to be $\nicefrac 13(1-p_c(\nesearrows))$. Next choose 
$$p_B= \eta^4(1-\eta)\inf\{\mathbb{P}_{p_c}(\mathcal A_n):n\ge1\}$$
which is strictly positive by the Russo-Seymour-Welsh theorem (see
\cite{Russo78,Seymour_Welsh} again). Define $p_G=1-2\eta$ (so still $p_G>p_c(\nesearrows)$).
Theorem~\ref{thm:main input} proves the existence of $\rho' < 1$ such
that oriented percolation in $\nesearrows$ with density of good lines
$\rho'$ and probability $p_G$ and $p_B$ in good and bad lines
respectively percolates a.s.

Next use the theorem of Liggett, Schonmann and Stacey 
\cite[Theorem 0.0]{Liggett_Schonmann_Stacey} to find some $\sigma$ such
that any 1-dependent family of variables $\{X_i:i\in V(\nesearrows)\}$
with $\mathbb P(X_i=1)>\sigma$ stochastically dominates
$(1-\eta)$-Bernoulli independent variables. Use the theorem again to find
some $\bar{\rho}$ such that any 1-dependent family of variables
$\{G_i:i\in \mathbb Z\}$ with $\mathbb P(G_i)>\bar{\rho}$ stochastically
dominates $\rho'$-Bernoulli independent variables.

Finally, we claim that for $n$ sufficiently large (depending on the $\varepsilon$ and $\rho$ from the statement of the theorem), the probability
that a column is good is more than $\bar{\rho}$, while the probability
that $\{X(v)= 1\} = \mathcal  A_n(2nv)$ occurs in a good column is more than
$\sigma$. Indeed, the first follows from Lemma \ref{l:good} while the
second follows from \eqref{e:spread3}. Fix $n$ to satisfy this
property. Finally, use continuity to choose some $q<p_c$ such that 
$\mathbb P_{p,q}^\Lambda(X(v)=1)>\sigma$ in any good column.

With all parameters defined, let us start with the columns. The
definitions of $n$ and $\rho$ allow to define $\rho'$-independent
variables $\Xi(i)$, depending only on $\Lambda$, such that if $\Xi(i)=1$ then the column $c(i)$ is
good. It will be convenient to define, for a vertex $v$ in a column
$c(i)$, $\Xi(v)=\Xi(i)$. Fix one realisation of $\Lambda$ and $\Xi$.

By the choice of $n$ and $q$, we know that for every $v$, 
$\Xi(v)=1\Rightarrow\mathbb P_{p,q}^\Lambda(X(v)=1)>\sigma$. Since these events are 1-dependent,
they dominate $(1-\eta)$-independent variables. Therefore the
variables $\{W(v):\Xi(v)=1\}$ dominate
$(1-\eta)^2$-independent variables. When $\Xi=0$ we use
Lemma \ref{l:domW} and get that, for any realisation of $W$ on the
$\{\Xi=1\}$, $\{W(v):\Xi(v)=0\}$ dominates
i.i.d.\ Bernoulli variables with probability 
\[
\eta^4\mathbb P(W(0)=1)=\eta^4(1-\eta)\mathbb P(X(0)=1)\ge p_B.
\]
All in all we get that $W$ dominates a family of independent Bernoulli random variables with mean $p_G$ where $\Xi=1$ and $p_B$ where $\Xi=0$. 
Denote a realisation of these independent variables by
$\Psi=\{\Psi(v):v\in\nesearrows\}$.

We get that $\Xi$ and $\Psi$ have exactly the distribution of variables on
$\nesearrows$ such that $\Xi$ are independent $\rho'$-Bernoulli random variables and $\{\Psi(v) =1 \}$ has
probability $p_G$ if $\Xi(v)=1$ and $p_B$ if $\Xi(v)=0$. Hence
Theorem \ref{thm:main input} applies and we get that the $\Psi$ percolate.
Since $X(v)\ge \Psi(v)$ so do the $\{X(v);\, v\in \nesearrows\}$. But if $X(v)=1$, then
$\mathcal A_n(2nv)$ occurs, and the geometric setup (see Figure
\ref{fig:L}) requires that these loops connect to one infinite
cluster, proving the theorem.
\end{proof}

\section*{Acknowledgements} This work was conducted during visits to the Weizmann Institute and the University of Geneva. We thank both institutions for their hospitality. H.D-C.\ was supported by the FNS and the NCCR SwissMap. M.R.H.\ was supported by the Brazilian CNPq grant 248718/2013-4 and by NCCR SwissMAP, the ERC AG COMPASP, the Swiss FNS. The research of G.K.\ was supported by the Israel Science Foundation and the Jesselson Foundation. The research of V.S. was supported in part by Brazilian CNPq grants 308787/2011-0 and 476756/2012-0 and FAPERJ grant E-26/102.878/2012-BBP. This work was also supported by  ESF RGLIS grant.

\end{document}